\documentclass{article}

\usepackage[utf8]{inputenc}
\usepackage{amssymb, amsmath, amsthm}
\usepackage[title]{appendix}
\usepackage{tikz}
\usepackage{pgfplots}
\usepackage{hyperref}

\newtheorem{lemma}{Lemma}[section]
\newtheorem{remark}[lemma]{Remark}

\newtheorem{proposition}[lemma]{Proposition}

\newtheorem{corollary}[lemma]{Corollary}
\theoremstyle{remark}

\DeclareMathOperator{\sign}{sign}
\newcommand{\real}{\mathbb{R}}
\newcommand{\poly}{\mathbb{P}}
\newcommand{\dualp}[1]{\left\langle #1 \right\rangle} 

\newcommand{\param}{\mathcal{P}}
\newcommand{\indexset}{\mathcal{I}}
\newcommand{\interpol}{\mathcal{I}}
\newcommand{\transforms}{\mathcal{T}}
\newcommand{\banach}{\mathcal{B}}

\begin{document}

\title{Transformed Snapshot Interpolation with High Resolution Transforms}
\author{G. Welper\footnote{Department of Mathematics, University of Central Florida, Orlando, FL 32816, USA, email \href{mailto:gerrit.welper@ucf.edu}{\texttt{gerrit.welper@ucf.edu}}. \newline 
This material is based upon work supported by the National Science Foundation under Grant No. 1912703.}}

\date{}
\maketitle

\begin{abstract}
  In the last few years, several methods have been developed to deal with jump singularities in parametric or stochastic hyperbolic PDEs. They typically use some alignment of the jump-sets in physical space before performing well established reduced order modelling techniques such as reduced basis methods, POD or simply interpolation. In the current literature, the transforms are typically of low resolution in space, mostly low order polynomials, Fourier modes or constant shifts. In this paper, we discuss higher resolution transforms in one of the recent methods, the transformed snapshot interpolation (TSI). We introduce a new discretization of the transforms with an appropriate behaviour near singularities and consider their numerical computation via an optimization procedure.

\end{abstract}

\smallskip
\noindent \textbf{Keywords:} Parametric PDEs, shocks, transformations, interpolation, convergence rates, optimization

\smallskip
\noindent \textbf{AMS subject classifications:} 41A46, 41A25, 35L67, 65M12

\section{Introduction}

On important ingredient in reduced order modeling and PDEs with random coefficients is the numerical approximation of parametric functions $u(x, \mu)$ with physical variables $x \in \Omega \subset \real^d$ and deterministic or random parameters $\mu$ in some parameter space $\param \subset \real^n$. The literature provides a rich toolbox for this kind of approximation, see e.g. \cite{CohenDeVore2015,BennerGugercinWillcox2015,QuarteroniManzoniNegri2015,HesthavenRozzaStamm2015} for an overview. However, if $u(x, \mu)$ has parameter dependent jumps or kinks, the vast majority of the available methods suffer from low convergence rates \cite{ConstantineIaccarino2012,OhlbergerRave2016,Welper2017}. As a result, problems from parametric hyperbolic PDEs, elliptic PDEs with jumping diffusion (in parameter dependent locations) and parametric level-set methods still pose severe challenges. 

Besides a large body of literature on stability, offline/online decompositions and error estimators for hyperbolic or singularly perturbed problems \cite{ChenGottliebHesthaven2005,HaasdonkOhlberger2008,HaasdonkOhlberger2008a,NguyenRozzaPatera2009,PetterssonAbbasbIaccarinoEtAl2010,PulchXiu2012,TryoenMaitreErn2012,MishraSchwab2012,DespresPoeetteLucor2013,YanoPateraUrban2014,PetterssonIaccarinoNordstroem2014,PacciariniRozza2014,DahmenPleskenWelper2014,Dahmen2015, AbgrallAmsallem2015,JinXiuZhu2016,BrunkenSmetanaUrban2018}, in the last few years, several groups have addressed the poor regularity of $u(x, \mu)$ and developed several methods with drastically improved convergence behaviour. While some use localization strategies \cite{ConstantineIaccarino2012,Peherstorfer2018}, currently the majority relies on additional transforms, shifts or similar tools in order to alleviate the issue. All these approaches have in common that they try to align the discontinuities or sharp gradients in physical space before applying classical reduced order modelling techniques as interpolation, POD or reduced bases. This alignment drastically improves the smoothness \cite{LiLiuShu2018}, Kolmogorov $n$-width, or decay rate of the singular values, respectively, which results in considerable efficiency gains in the presence of jumps. Some approaches use characteristics \cite{TaddeiPerottoQuarteroni2015}, Lagrangian formulations \cite{MojganiBalajewicz2017}, extra equations or conditions \cite{GerbeauLombardi2012,OhlbergerRave2013,GerbeauLombardi2014} or jump tracking \cite{SchulzeReissMehrmann2018}. Another approach is to use displacement interpolation or  optimal transport \cite{IolloLombardi2014,RimMandli2018,RimMandli2018a}, which naturally builds on transport along the physical variables. One of the more common recent approaches uses a composition of a snapshot $u(x - s, \mu)$ with a shift $s$ or alternatively a transform $u(X(x),\mu)$ in order to align the jump locations in parameter, where $s$ or $X$ usually depend on the parameter and may or may not depend on the location $x$. These shifted or transformed snapshots are then used for interpolation, reduced bases or further compressed by POD, see e.g. \cite{Welper2017,ReissSchulzeSesterhenn2015,CagniartMadayStamm2019,CagniartCrisovanMadayEtAl2017,NairBalajewicz2017,Welper2017a}. Alternatively,  in \cite{RimMoeLeVeque2017}, instead of a function composition, the authors shift the coefficients of a discretization.

For the latter class of methods, one needs numerical procedures to find the shifts $s$ or transforms $X$. Since the shifts or transforms generally depend on $x$ and $\mu$, one first chooses a suitable discretization. Then, one minimizes an error formula among the chosen discrete representation during the ``offline phase'', either between individual (transformed) snapshots or the full reconstruction error of the method. It seems that in all of the current literature the discretizations are of low order in physical space either constants, low order polynomials or Fourier modes. Clearly, these limitations on the transform's resolutions pose severe limitations on their applicability to many practical problems. As we will see in Section \ref{sec:transform-example}, already quite simple scenarios require transforms with high resolution in physical space.  

As a first part of this paper, we discuss possible discretizations of high resolution transforms, which in the most general case are functions $X(x,\mu)$ depending on the physical variables $x$ and an arbitrary parameter $\mu$. As we will discuss in Section \ref{sec:transform-example}, these functions share some major properties with $u$ itself: They may require high resolution in the spacial variables and can be non-smooth in parameter. However, we cannot use the same discretization strategy as for $u$ without introducing a second layer of transforms. Instead, we will see that these transforms are naturally given as solutions of ODEs (generally different from characteristic ODEs) and that variants of implicit solvers thereof provide efficient ways for their discretization.

In principle, one can then use the same optimization strategies as for low resolution transforms, but as we will see in some experiments below, this results in poor outcomes. Typically, naive implementations of gradient descent optimizers stall after a few steps and the resulting transform is visually indistinguishable from the initial transform. In order to overcome this problem, in Section \ref{sec:optimize}, we discuss that the transforms $X$ are functions of $x$ and must be considered in appropriate function spaces. This entails that the Fr\'{e}chet derivative of the objective function to optimize the transform is only a functional in a dual space and must be appropriately ``lifted'' by a Riesz map to obtain a well defined gradient descent method. In this paper, we report on two lifting methods, one by inverting a Laplacian and one by using a multi-level frame.

The paper is organized as follows: To keep it self contained, in Section \ref{sec:tsi-review}, we briefly review the transformed snapshot interpolation (TSI), which is one of the recently introduced transformed based reduced order modeling techniques for problems with jumps. Then, in Section \eqref{sec:transform-example}, we consider an example problem to better understand the requirements for high resolution transforms. In Section \ref{sec:ode} we discuss the discretization of the transforms and in Section \ref{sec:stability} consider their stability. In Section \ref{sec:optimize}, we introduce suitable optimizers to practically find the transforms and finally in Section \ref{sec:experiments} we provide some numerical experiments.

\section{Transformed Snapshot Interpolation}
\label{sec:tsi-review}

In order to keep the paper self contained, let us first briefly review the \emph{transformed snapshot interpolation} (TSI) introduced in \cite{Welper2017}. It approximates a parametric function $u(x, \mu)$ with physical variables $x \in \Omega \subset \real^d$ and deterministic or random parameters $\mu \in \param \subset \real^n$ by 
\begin{equation}
  u(x, \mu) \approx u_m(x, \mu) := u_m(x, \mu; X) := \sum_{\eta \in \param_m} \ell_\eta(\mu) u(X(\eta; \mu, x), \eta),
  \label{eq:tsi}
\end{equation}
where $\ell_\eta(\mu)$ are Lagrange basis polynomials with respect to given interpolation points $\eta \in \param_m \subset \param$. Similar to reduced basis methods and PODs, this interpolation only requires us to know snapshots $u(\cdot, \eta)$ at finitely many interpolation points in parameter. In addition we need the transforms $\eta \to X(\eta; \mu, x) \subset \Omega$ which can be regarded as curves in physical space, together with an initial value $X(\mu; \mu, x) = x$. We often write $X(\eta) = X(\eta;\mu, x)$ if $\mu$ and $x$ are understood from context.

As a notational convention, both $\mu$ and $\eta$ refer to parameters, where $\mu$ usually denotes a target parameter where we want to approximate $u(\cdot, \mu)$ and $\eta$ is either a continuous auxiliary parameter in $\param$ as in the curve $\eta \to X(\eta)$ or a discrete ``source'' parameter in $\param_m$ at locations where we know the snapshots, e.g. as in the definition of the TSI \eqref{eq:tsi}. The distinction between continuous or discrete should always be clear from context.

Let us now come back to the TSI \eqref{eq:tsi}. Choosing the trivial transform $X(\eta; \mu, x) = x$, the formula reduces to a standard polynomial interpolation. However, this simple approach does not achieve good convergence rates in case $u(x, \mu)$ has jumps or kinks in parameter dependent locations, because these entail that $u$ is not even differentiable in the interpolation direction. The extra transform $X(\eta; \mu, x)$ is used to align these jumps so that they become ``invisible'' to the interpolation. In other words, for a target parameter $\mu$ and interpolation point $\eta$, the transform $X(\eta; \mu, x)$ ensures that the \emph{transformed snapshots} $(x, \eta) \to u(X(\eta; \mu, x), \eta)$ have jumps in the same locations as the correct solution $u(\cdot, \mu)$, independent of $\eta$. In particular, for all $x$, which are not in the jump-set of $u(\cdot, \mu)$, the transformed snapshots are smooth in $\eta$. The TSI \eqref{eq:tsi} is merely a polynomial interpolation of these transformed snapshots in $\eta$, and thus provides high order accuracy.

Note that in the TSI \eqref{eq:tsi} the transforms $X(\eta; \mu, x)$ are only evaluated at finitely many interpolation points $\eta \in \param_m$ and the continuous variables $x \in \Omega$ and $\mu \in \param$. If the transforms are smooth in $\mu$, they can be efficiently discretized in $\mu$ by interpolation
\begin{equation}
  X(\eta; \mu, x) = \sum_{\gamma \in \param_m} \ell_\gamma(\mu) X(\eta; \gamma, x),
  \label{eq:transform-interpol}
\end{equation}
called \emph{low resolution transforms} in the following. This was the original choice in \cite{Welper2017} together with low order polynomials for the remaining finitely many functions $x \to X(\eta; \gamma, x)$, $\eta, \gamma \in \param_m$ with zero normal flux as a numerical substitute for the condition $X(\eta; \mu, x) \in \Omega$. As we will see e.g. in Section \ref{sec:transform-example}, neither the smoothness in $\mu$ nor the choice of low order polynomials in $x$ is warranted in some fairly simple examples. Therefore, we will discuss alternative representations of the transforms in Section \ref{sec:ode}.

Finally, similar to neural network training or greedy algorithms for reduced basis methods, the remaining transforms $X(\eta; \gamma, x)$, $\eta, \gamma \in \param_m$ are trained by minimizing the worst case error
\begin{equation}
  \sup_{\mu \in \param_T}\|u(x, \mu) - u_m(x, \mu; X)\|_{L_1(\Omega)},
  \label{eq:opt}
\end{equation}
Since we do not know $u(\cdot, \mu)$ for all $\mu \in \param$, we cannot calculate the full error and confine ourselves to a training sample $\param_T \subset \param$. Of course other $L_p$-norms in parameter are also possible. The $L_1$-norm for the physical variables is natural for hyperbolic problems and ensures that the objective function of this optimization problem is differentiable almost everywhere, given that $u$ is of bounded variation, see \cite{Welper2017} for more details.

\section{An Example Problem}
\label{sec:transform-example}

In this section, we first consider a simple example problem that highlights some typical difficulties for the low resolution transforms \eqref{eq:transform-interpol}. Namely, let us consider the function
\begin{equation}
  u(x, \mu) = \left\{ \begin{array}{rl} 
    1 & x \le -(1 - \mu) \\
    -1 & x \ge \phantom{-(} 1 - \mu \phantom{)} \\
    0 & \text{else}
  \end{array}\right., 
  \label{eq:example-func-1}
\end{equation}
for $\mu < 1$ and 
\begin{equation}
  u(x, \mu) = \left\{ \begin{array}{rl} 
    1 & x \le 0 \\
    -1 & x > 0
  \end{array}\right., 
  \label{eq:example-func-2}
\end{equation}
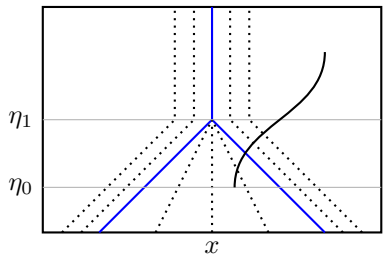
\begin{figure}[htb]

  \centering
  \begin{tikzpicture}[scale=1.5]

  \draw[thick] (-1.5,0) rectangle (1.5,2);

  \draw[thick, blue] ( 1,0) -- (0,1);
  \draw[thick, blue] (-1,0) -- (0,1);
  \draw[thick, blue] ( 0,1) -- (0,2);

  \draw[thick, dotted] (-1.16, 0) -- (-0.16, 1) -- (-0.16, 2);
  \draw[thick, dotted] (-1.33, 0) -- (-0.33, 1) -- (-0.33, 2);
  \draw[thick, dotted] ( 1.16, 0) -- ( 0.16, 1) -- ( 0.16, 2);
  \draw[thick, dotted] ( 1.33, 0) -- ( 0.33, 1) -- ( 0.33, 2);
  \draw[thick, dotted] (-0.5, 0) -- (0,1);
  \draw[thick, dotted] ( 0.5, 0) -- (0,1);
  \draw[thick, dotted] ( 0  , 0) -- (0,1);

  \node[below] at (0,0) {$x$};

  \node[left] at (-1.5,0.4) {$\eta_0$};
  \draw[black!30] (-1.5, 0.4) -- (1.5, 0.4);
  \node[left] at (-1.5,1) {$\eta_1$};
  \draw[black!30] (-1.5, 1) -- (1.5, 1);

  \draw[thick] (0.2, 0.4) to[out=90, in=-90] (1.0, 1.6);
  
  \end{tikzpicture}

\caption{Blue: Jump set of the example function \eqref{eq:example-func-1}, \eqref{eq:example-func-2} in the $x-\mu$-plane. Dotted lines: Curves $\eta \to X(\eta; \mu, x)$ defined in \eqref{eq:example-transform}. Black: Inappropriate transform $\eta \to X(\eta; \mu, x)$ crossing a jump.}

  \label{fig:jump-set}
\end{figure}
for $\mu \ge 1$, depicted in Figure \ref{fig:jump-set}. It serves as a prototype of problems for which the original TSI as introduced in Section \ref{sec:tsi-review} is problematic: It starts out with two jumps for $\mu < 1$ that collide into one single jump for $\mu \ge 1$. This poses a problem for the TSI because we cannot properly align a snapshot with a single jump with a snapshot with two jumps. More formally, the TSI interpolates the transformed snapshots $\eta \to u(X(\eta; \mu, x), \eta)$, which should ideally be smooth in the $\eta$ variable. Since $\eta \to X(\eta; \dots)$ itself is smooth, this can only be true if $X(\eta; \mu, x)$ is not contained in the jump-set of $u(\cdot, \eta)$ for all $\eta \in \param$. Obviously this is impossible for a curve starting between the two jumps in the lower half of Figure \ref{fig:jump-set} indicated by the black solid line crossing a blue jump. In fact, this condition forces all curves $\eta \to X(\eta; \mu, x)$, which  start at $\mu, x$ between the two jumps, to converge to the collision point for $\eta \to 1$, indicated by the dotted lines. The remaining dotted lines complete this to a reasonable transform given by 
\begin{align}
  X(\eta; \mu, x) = \left\{ \begin{array}{rl} 
    x-(\mu-\eta) & x \le -(1 - \mu) \\
    x+(\mu-\eta) & x \ge \phantom{-(} 1 - \mu \phantom{)} \\
    x - x \frac{\mu-\eta}{\mu-1} & \text{else}
\end{array}\right..
\label{eq:example-transform} 
\end{align}
In particular, if we confine our parameter range to $\mu < 1$ or $\mu \ge 1$, these transforms never cross jumps, corresponding to the parameter ranges where $u$ has either one or two jumps. Therefore, we may proceed as follows: We subdivide the parameter domain into two pieces for $\mu \lessgtr 1$ and apply a TSI on each of them. As simple version of this idea, which finds partitions of $\param$ via $h$ or $hp$ adaption, is discussed in \cite{Welper2017a}. A more sophisticated algorithm that cuts the parameter domain precisely at the jump collision is in preparation. In this paper, we do not deal with the proper subdivision but with the necessarily singular transforms:

\begin{enumerate}

  \item For $\mu = 1$ and $\eta < 1$, the transform has a jump at $x=0$. 

  \item The dependence of the $\mu$ variable is of the type $\sim 1/(\mu-1)$ in the regions between the two jumps and therefore singular at $\mu = 1$.

\end{enumerate}

Both problems cannot be handled with the low resolution transforms \eqref{eq:transform-interpol}. These were low order polynomials both in the $x$ and $\mu$ variables. For the $x$ variable one can simply increase the spacial resolution of the transforms to the same resolution of the snapshots. Since the latter must be stored anyways, the overall degrees of freedom do not blow up significantly. However, this higher resolution requires some extra care with regard to the optimization of the error \eqref{eq:opt}, as is discussed in Section \ref{sec:optimize}.

With regard to the singularity in the $\mu$ variable, a simple increase of the resolution is more challenging without increasing the number of snapshots and therefore we discuss an alternative discretization of the transform in Section \ref{sec:ode}.

Although the example discussed in this section is quite simple, the arising problems are quite typical for problems where the jump-sets cannot be properly aligned between snapshots. As we will see later in Section \ref{sec:ode} when we have developed some better tools, singular behaviour of the transforms generally occurs when the jump-set topology changes in parameter.

\section{Transforms as solutions of ODEs}
\label{sec:ode}

In this section, we consider discretizations of the transforms, which allow the types of singular behaviour that we observed in our introductory example in Section \eqref{sec:transform-example}. We have already chosen a notation for the transforms $\eta \to X(\eta) = X(\eta; \mu, x)$ with ``initial value'' $X(\mu, \mu, x) = x$ that suggests to define them as a solution of an ODE
\begin{align}
  \frac{d}{d \eta} X(\eta) & = \Phi(X(\eta), \eta), &
  X(\mu) = x,
  \label{eq:tode}
\end{align}
where we call $\Phi: \Omega \times \param \to \real$ the \emph{transport field}. First note that this definition restricts us to one dimensional parameter domains. Nonetheless, higher dimensional parameters can be treated by a tensor product type construction as shown in \cite{Welper2017a}, where we split dimensions of the parameter space one by one. Even for one dimensional parameters, defining transforms via ODEs restricts our possible choices since it induces a semi-group property in the dynamic variable of the ODE: $X(\eta; \mu, x) = X(\eta; \nu, X(\nu; \mu,x))$. However, this restriction does not seem to be overly severe and has two major benefits: It was argued in \cite{Welper2017} that it greatly helps optimizers to find global minima of \eqref{eq:opt} and as we shall see in this section, it provides a way to find discretizations of the transforms in view of singularities. 

Let us first review the example in Figure \ref{fig:jump-set} again. In order to properly align the jumps, the blue lines must be a solution of the ODE \eqref{eq:tode}, defining the transforms. This implies that the trajectories intersect at the jump collision so that the ODE does not have a unique solution. Therefore, at least one of the assumption of the Picard Lindel\"{o}f theorem cannot be valid, which is typically the Lipschitz continuity of $\Phi$. This provides another motivation for dealing with high spacial resolution transforms.

One can easily generalize this observation: If the jump-sets of two snapshots $u(\cdot, \mu)$ and $u(\cdot, \eta)$ are not homeomorphic, no ODE, which aligns the jumps, can satisfy all assumptions of the Picard Lindel\"{o}f theorem. For if one would, the solution $x \to X(\eta; \mu, x)$ would be a homeomorphism of the respective jump-sets, contradicting the assumption of having a change in topology.

In order to find a numerical method based on the ODE \eqref{eq:tode}, let us assume for the moment that we know $\Phi(x, \mu)$ for all possible values of $x$ and $\mu$. Then we still have to solve the ODE \eqref{eq:tode} to find $X(\eta)$ and ultimately the TSI interpolation. The easiest choice that comes to mind is Euler's method
\[
  X(\eta) \approx x + (\eta - \mu) \Phi(x, \mu),
\]
which has the drawback that we have to be able to evaluate $\Phi$ at every possible target parameter $\mu$, which are of course unknown at the time we compute the transforms. If however we use the implicit Euler method, we have
\begin{equation}
  X(\eta) \approx x + (\eta - \mu) \Phi(X(\eta), \eta),
  \label{eq:impl-euler}
\end{equation}
where $\Phi$ is only evaluated at the parameter $\eta$, which for the TSI reconstruction is only required at the finitely many snapshot locations $\param_m$.

Of course the implicit Euler method limits trajectories $X(\eta)$ to lines, so that in the remainder of this section, we construct higher order alternatives. First note that we cannot use higher order ODE solvers out of the box: Runge Kutta methods require knowledge of $\Phi(\cdot, \eta)$ at intermediate parameters depending on the unknown target $\mu$ and multi-step methods rely on single step methods to get started.

Given these difficulties, we take an alternative but related route. Assume that we know $\Phi(\cdot, \eta)$ at the interpolation points $\eta \in \param_n$, just as we know the snapshots $u(\cdot, \eta)$ at all $\param_m$. Here $\param_n$ is another set of finitely many nodes, which can in principle differ from $\param_m$. Restricting the ODE \eqref{eq:tode} to the information we have, we obtain
\begin{align}
  X'(\eta) & = \Phi(X(\eta), \eta) \text{ for } \eta \in \param_n & X(\mu) & = x,
  \label{eq:transform-interpolation-conditions}
\end{align}
As we shall see in Proposition \ref{prop:implicit-transform} below, these conditions are sufficient to uniquely determine the curve $X(\eta)$, if we restrict them to be polynomials of degree $n = |\param_n|$, denoted by $\poly^n$. The resulting transforms $X(\eta)$ are called \emph{high resolution transforms} in the following. 

Let us first verify that this construction contains the implicit Euler method as a special case. For one single interpolation point, the transform of \eqref{eq:transform-interpolation-conditions} is a line and therefore uniquely determined by the initial condition $X(\mu) = x$ and its slope, $X'(\eta) = \Phi(X(\eta), \eta)$ at the interpolation point $\eta$. This implies that the curve is given by
\[
  X(\xi) = x + (\xi - \mu) \Phi(X(\eta), \eta)
\]
at any point $\xi \in \param$, which corresponds the implicit Euler method for $\xi = \eta$. 

Similar to ODE theory, the next proposition provides some conditions under which the interpolation problem \eqref{eq:transform-interpolation-conditions} has a unique solution.

\begin{proposition}
\label{prop:implicit-transform}
Let $I$ be an interval that contains the interpolation points $\param_n$ and the initial point $\mu$.

\begin{enumerate}

  \item \label{item:global} Assume that all $\Phi(\cdot, \eta): \real^d \to \real^d$, $\eta \in \param_n$ are bounded and continuous. Then the interpolation problem \eqref{eq:transform-interpolation-conditions} has a solution in $\poly^n$ restricted to $I$.

  \item \label{item:local} Assume that all $\Phi(\cdot, \eta): \real^d \to \real^d$, $\eta \in \param_n$ are Lipschitz continuous with Lipschitz constant $L$ and that $\Lambda_{n-1} L |I| < 1$, where $\Lambda_{n-1} := \max_{\mu \in I} \sum_{\eta \in \param_n} |\ell_\eta(\mu)|$ is the Lebesgue constant for the interpolation points $\param_n$. Then the interpolation problem \eqref{eq:transform-interpolation-conditions} has a unique solution in $\poly^n$ restricted to $I$.

\end{enumerate}

\end{proposition}

Note that except for the discrete locations of the jump-set topology changes, the maps $\Phi(\cdot, \eta)$ can be expected to be continuous. Close to the toplogy changes, the Lipschitz constant will deteriorate, which is discussed at the end of this Section. In addition, we are only interested in $\Phi(x, \eta)$ with $x \in \Omega$, so that the boundedness assumption can be met by suitable modifications of $\Phi$ outside of $\Omega$. Therefore, with rather mild conditions the first part of the Proposition ensures the global existence of transforms which satisfy \eqref{eq:transform-interpolation-conditions}.

As argued at the beginning of this section, we expect non-unique solutions at the jump-set topology changes for the ODE \eqref{eq:tode}. Likewise, we cannot expect global uniqueness for the discrete variant \eqref{eq:transform-interpolation-conditions} either. Nonetheless, locally this is possible by classical arguments as stated by the second part of the proposition.

The proof is essentially a discrete version of the Picard-Lindel\"{o}f theorem. To see this more clearly and for later reference, by the uniqueness of polynomial interpolation, the problem \eqref{eq:transform-interpolation-conditions} is equivalent to to following continuous variant: find $x \in \poly_n$ such that
\begin{align}
  X'(\eta) & = [\interpol_{n-1} \Phi(X(\cdot), \cdot)](\eta) \text{ for } \eta \in \param, & X(\mu) & = x,
  \label{eq:transform-interpolation-conditions-continuous}
\end{align}
or equivalently
\begin{equation}
  X(\xi) = x + \int_\mu^\xi [\interpol_{n-1} \Phi(X(\cdot), \cdot)](\gamma) \, d \gamma =: F[X](\xi),
  \label{eq:ode-poly-fp}
\end{equation}
where $\interpol_{n-1}$ is the interpolation operator on the $n$ points $\param_n$. Note that this is not a standard ODE because the right hand side depends $X(\gamma)$ for all $\gamma \in \param_n$ instead of $X(\eta)$ at the single point $\eta$. Nonetheless, the Picard-Lindel\"{o}f theorem carries over verbatim. 

\begin{proof}

  We use the usual fixed point argument applied to the fixed point function $F$ defined in \eqref{eq:ode-poly-fp} on the space $\poly_n(I)$ of polynomials up to degree $n$, restricted to the interval $I$ and equipped with the norm $\|X\| := \max_{\gamma \in I} |X(\gamma)|$. Due to the interpolation $\interpol_{n-1}$ in the definition of $F$ and the subsequent integration $F$ maps $\poly_n(I)$ to itself. 

We first show Part \ref{item:global} of the proposition. Instead of the usual contraction mapping principle, we exploit the finite dimensional nature of $\poly_n$ and use Schauder's fixed point theorem. To this end, note that the continuity of $\Phi(\cdot, \eta)$ implies the continuity of $F$. In addition, we have
\begin{align*}
  \|F[X]\| & = \max_{\xi \in I} \left| x + \int_\mu^\xi [\interpol_{n-1} \Phi(X(\cdot), \cdot)](\gamma) \, d \gamma \right| \\
  & \le |x| + \Lambda_{n-1} \|\Phi(X(\cdot), \cdot)\| \int_\mu^\xi \, d \gamma \\
  & \le |x| + \Lambda_{n-1} M |I|,
\end{align*}
where $M$ is the upper bound of $|\Phi(\cdot, \eta)|$, $\eta \in \param_n$ and $\Lambda_{n-1}$ is the Lebesgue constant for the interpolation $\interpol_{n-1}$. It follows that $F$ maps the ball of all polynomials $p \in \poly_n$ with $\|p\| \le |x| + \Lambda_{n-1} M |I|$ to itself. In conclusion, Schauder's fixed point theorem implies the existence of a fixed point of the integral equation \eqref{eq:ode-poly-fp} and thus the discrete definition \eqref{eq:transform-interpolation-conditions} of the transforms.

Since Schauder's fixed point theorem does not guarantee uniqueness, we continue to proof part \ref{item:local} of the proposition using the classical contraction mapping principle. To show that $F$ is a contraction, first note that
\begin{align*}
  \left\| \interpol_{n-1} \Phi(X(\cdot), \cdot) - \interpol_{n-1} \Phi(Y(\cdot), \cdot) \right\|
  & \le \Lambda_{n-1} \max_{\eta \in \param} \left| \Phi(X(\eta), \eta) - \Phi(Y(\eta), \eta)\right| \\
  & \le \Lambda_{n-1} L \|X - Y\|,
\end{align*}
where again $\Lambda_{n-1}$ is the Lebesgue constant and $L$ the Lipschitz constant of $\Phi$. It follows that
\begin{equation}
    \|F[X] - F[Y]\|
    \le \Lambda_{n-1} L \left\| X - Y \right\| \int_\mu^\xi \, d \gamma
    \le \Lambda_{n-1} L |I| \left\| X - Y \right\|.
  \label{eq:contraction}
\end{equation}
Since $\Lambda_{n-1} L |I| < 1$, the proposition follows from Banach's fixed point theorem.

\end{proof}

\begin{remark}
  Since $X(\eta)$ is a polynomial, which is uniquely defined on the interval $I$, we may extend it to the real line $\real$ still satisfying the discrete and continuous interpolation problems \eqref{eq:transform-interpolation-conditions} and \eqref{eq:transform-interpolation-conditions-continuous}. However, beyond the interval $I$ in Proposition \ref{prop:implicit-transform} part \ref{item:local}, the solution is not necessarily unique in the sense that there can be two solutions $x$ and $y$ of \eqref{eq:transform-interpolation-conditions} with different initial values $X(\mu) = x$ and $Y(\mu) = y$ that intersect outside of $I$.
\end{remark}

Let us compare the new transform $X(\eta)$ constructed in the above Proposition with the original choice \eqref{eq:transform-interpol} defined via polynomial interpolation. The original choice is polynomial in $\mu$, whereas the ``ideal'' transform in the example of Section \ref{sec:transform-example} is rational in $\mu$. In contrast, both the new transform and the ``ideal'' one are polynomial in $\eta$. In addition, the stability function of the implicit Euler method is a rational function of $\eta-\mu$ and thus for a fixed interpolation node $\eta$ and variable target $\mu$ a rational function of $\mu$. Again, this matches the ``ideal'' transform of the example.

For a practical application of the new transform, we need an algorithm to solve the interpolation problem \eqref{eq:transform-interpolation-conditions}. 
Since the fixed point iteration for the map $F: \poly_n \to \poly_n$ in the proof of Proposition \ref{prop:implicit-transform} is defined on polynomials, it can be carried out practically and it remains to devise an algorithm to compute $Y = F[X] \in \poly_n$. This is reminiscent of implicit ODE solvers as e.g. the implicit Euler method \eqref{eq:impl-euler}, which are also often solved by fixed point methods. We carry out the fixed point iteration in a Newton basis, which easily leads to an algorithm that is sufficient for the first experiments presented in this paper, although more clever approaches might be possible.

Although we integrate polynomials, we cannot readily apply standard quadrature rules because the integration bounds are variable and not known beforehand. As an alternative, we first differentiate the definition \eqref{eq:ode-poly-fp} of $F$ to obtain
\begin{equation}
  Y'(\xi) := \frac{d}{d \xi} F[X](\xi) = \left[ \interpol_{n-1} \Phi(X(\cdot), \cdot) \right](\xi)
  \label{eq:interpol-ode}
\end{equation}
The derivative $Y'(\xi)$ is a polynomial in $\poly_{n-1}$ and thus completely specified by its values at the interpolation points $\eta \in \param_n$, which are
\[
  Y'(\eta) = \left[ \interpol_{n-1} \Phi(X(\cdot), \cdot) \right](\eta) = \Phi(X(\eta), \eta),
\]
Of course these are the same interpolation conditions as in the original problem \eqref{eq:transform-interpolation-conditions}, with the difference that currently $X$ is not necessarily a fixed point. Now, let us represent $Y$ in a Newton basis with respect to the interpolation points $\{\mu\} \cup \param_n = \{\mu, \eta_1, \dots, \eta_n\}$, i.e.
\begin{align}
  Y(\xi) & = \sum_{i=0}^n a_i \omega_i(\xi), 
  & \omega_0(\xi) & = 1
  & \omega_i(\xi) & = (\xi-\mu) \prod_{j<i} (\xi-\eta_i).
  \label{eq:transform-newton-basis}
\end{align}
First note that choosing $\xi = \mu$ the definition of $F$ implies that $a_0 = x$. For the remaining $a_1, \dots, a_n$, we first differentiate $Y'(\xi) = \sum_{i=1}^n a_i \omega_i'(\xi)$ and plug in the interpolation points to obtain the system
\begin{align}
  \begin{pmatrix}
    \omega_1'(\eta_1) & \dots & \omega_n'(\eta_1) \\
    \vdots & & \vdots \\
    \omega_1'(\eta_n) & \dots & \omega_n'(\eta_n)
  \end{pmatrix}
  \begin{pmatrix}
    a_1 \\ \vdots \\ a_n
  \end{pmatrix}
  =
  \begin{pmatrix}
    \Phi(X(\eta_1), \eta_1) \\
    \vdots \\
    \Phi(X(\eta_n), \eta_n)
  \end{pmatrix}.
\label{eq:transform-coefficients}
\end{align}
Note that for Newton interpolation we have $\omega_j(\eta_i) = 0$ for $j>i$, which leads to triangular Vandermonde type matrices for standard interpolation problems. However, this does not imply that also $\omega_j'(\eta_i) = 0$ for $j>i$. Therefore, in our case we deal with a full matrix. Nonetheless, we can compute an LU decomposition beforehand and solve the system efficiently for every new right hand side.

Now that we can compute $Y=F[X]$, we can numerically solve the interpolation problem \eqref{eq:transform-interpolation-conditions} by the fixed point iteration
\begin{align}
  X^{k+1} & = F[X^k], & X^0 = x.
  \label{eq:fixed-point}
\end{align}
The following Corollary to Proposition \ref{prop:implicit-transform} states that this converges to a solution of the interpolation problem.

\begin{corollary}
  \label{cor:fixed-point}
  Assume that all assumptions of Proposition \ref{prop:implicit-transform} part \ref{item:local} are satisfied. Then the fixed point iteration \eqref{eq:fixed-point} converges to a solution of the interpolation problem \eqref{eq:transform-interpolation-conditions}.
\end{corollary}

\begin{proof} 
  
The corollary is a direct consequence of the proof of Proposition \ref{prop:implicit-transform} part \ref{item:local}.

\end{proof}

The fixed point iteration \eqref{eq:fixed-point} provides us with polynomials $\eta \to X(\eta; \mu, x)$ for $\eta \in \param$. For the TSI \eqref{eq:tsi}, we only need the transforms at the locations $\eta \in \param_m$, which we can easily evaluate even if $\param_m \ne \param_n$, i.e. we use different nodes for the transform field than for the snapshots.

As discussed earlier, we expect the Lipschitz constant $L$ to deteriorate near jump-set topology changes so that a plain fixed point iteration as in the last corollary is generally insufficient. Nonetheless, we can counterbalance this with a small interval length $I$ in order to meet the assumptions of Proposition \ref{prop:implicit-transform}. The obvious choice is to add more snapshots close to the singularity, but that clearly defeats the purpose of this section. Alternatively, we use this observation to create better initial values for the fixed point iteration. To this end, note that neither the definition of $F[X]$ nor the assumptions of Proposition \ref{prop:implicit-transform} need $\Phi(\cdot, \eta)$ for the full continuum of parameters $\eta$, but only for the few discrete $\eta \in \param_n$. For clarity, let us call them $\Phi_0, \dots, \Phi_n$. Once we have fixed these functions, the interval length $I$ only enters the equations via the Newton basis \eqref{eq:transform-newton-basis} for the interpolation points $\{\mu\} \cup \param_n$. Therefore, in order to obtain a shorter interval $I$, we rescale these interpolation points by
\begin{equation}
  \{\mu\} \cup \{ \mu + s(\eta - \mu) : \, \eta \in \param_n\}
  \label{eq:scaling}
\end{equation}
which match the original choice with $s=1$ and shrink to $\{\mu\}$ for $s=0$. For $s$ sufficiently small, the fixed point iteration \eqref{eq:fixed-point} will then be successful for the unchanged $\Phi_0, \dots, \Phi_n$. We use the outcome as an improved initial value for a new fixed point iteration with larger scaling $s$, although this case may no longer be covered by Corollary \ref{cor:fixed-point}. In the numerical experiments, we choose an increasing sequence of scaling factors $s_i$ and perform fixed point iterations \eqref{eq:fixed-point} for the corresponding scaled parameters, each using the last iterate from the previous scaling as initial value.

\section{A note about stability}
\label{sec:stability}

Since we are dealing with (near) singular jump-sets, we have to be careful about the stability of the TSI reconstruction with respect to perturbations of the transforms. Such perturbations are generally inevitable because of numerical errors due to iterative optimizers for \eqref{eq:opt} and discretizations of the transport fields $\Phi$. 

Let us first consider a heuristic motivation in $1d$. Assume we have two transforms $X_i(x) := X_i(\eta; \mu, x)$, $i=0,1$. Ignoring higher order terms, and using the substitution $y = X_0(x)$ we can informally estimate the perturbation error between two transformed snapshots $v \circ X_i(x) := u(X_i(x), \eta)$ by
\begin{equation}
  \begin{aligned}
    \|v \circ X_1 - v \circ X_0\|_{L_1(\Omega)} 
    & \approx \int_\Omega |u'(X_0(x))| |X_0(x) - X_1(x)| \, dx \\
    & = \int_\Omega |u'(y)| |y - X_1(X_0^{-1}(y))| \frac{1}{|X_0'(x)|} \, dy \\
    & \le \sup_{y \in \Omega} \left[ \frac{|y - X_1(X_0^{-1}(y))|}{|X_0'(x)|} \right] \int_\Omega |u'(y)|  \, dy \\
    & = \sup_{x \in \Omega} \left[ \frac{|X_0(x) - X_1(x)|}{|X_0'(x)|} \right] \|u\|_{BV(\Omega)}.
  \end{aligned}
  \label{eq:stability-informal}
\end{equation}
If $X_0'$ is bounded from above and below, this gives us an estimate of the type
\begin{equation*}
  \|v \circ X_1 - v \circ X_0\|_{L_1(\Omega)} 
  \le C \|X_0(x) - X_1(x)| \|_{L_\infty(\Omega)} \|u\|_{BV(\Omega)}.
\end{equation*}
for some constant $C$. However, this argument is not fully satisfactory for two reasons.
\begin{enumerate}
  \item In the vicinity of jump-set topology changes $X_0'(x)$ is typically not uniformly bounded from above and below.
  \item If $X_i$ has a jump or sharp gradient, reasonable perturbations may not be close in the $L_\infty$-norm.
\end{enumerate}
If we are more careful, we see that wherever $X_0$ has a large gradient so that we possibly have a large $L_\infty$ error of the transforms, the denominator of the right hand side of \eqref{eq:stability-informal} is also large. The argument does not carry over to sharp gradients of $X_1$, but with an analogous argument, we could have had $X_1'$ in the denominator instead. Therefore, in the following, we consider the question if a judicious choice of the denominator can at least in principle provide some control of the perturbation error in the face of sharp gradients.

The choice of the denominator, now for arbitrary spacial dimensions, is done by selecting a homotopy, i.e. a continuous transition, $X_s(\eta; \mu, x)$, $0 \le s \le 1$ between the two transforms so that the boundary cases $X_0(\eta; \mu, x)$ and $X_1(\eta; \mu, x)$ are our original two perturbed transforms. A simple candidate is a convex combination $X_s = (1-s)X_0 + s X_1$. The following stability result is a rigorous version of our heuristic argument above and a slight refinement from a proposition in \cite{Welper2017}. 

\begin{proposition}
  \label{prop:stab}
  Assume that $u \in BV(\Omega)$ and that
  \begin{equation*}
    \int_{X_s(\eta; \mu, \cdot)^{-1}(A)} \, dx  \le B |A|
  \end{equation*}
  for some $B \ge 0$, all measurable sets $A$, all $\eta \in \param_n$, $\mu \in \param$ and $s \in \{0,1\}$. Using $\dot{X} := \frac{d}{ds}X_s$, define
  \begin{equation}
    S(\mu, \eta) := \sup_{A} \frac{1}{|A|} \int_0^1 \int_{X_s(\eta; \mu, \cdot)^{-1}(A)} |  \dot{X}_s(\eta; \mu, x)| \, dx \, ds.
    \label{eq:stability-factor}
  \end{equation}
  Then, we have
  \begin{equation*}
    \|u_m(\cdot, \mu; X_0) - u_m(\cdot, \mu; X_1)\|_{L_1(\Omega)}
    \le \Lambda_n \max_{\eta \in \param_m} \big[ S(\mu, \eta) \|u(\cdot, \eta)\|_{BV(\Omega)} \big],
  \end{equation*}
  independent of $B$ and where $\Lambda_m$ is the Lebesgue constant.
\end{proposition}

The proof is essentially the same as in \cite{Welper2017} and given in Appendix \ref{appendix-proof-prop}. For sufficiently regular transforms $X_s$, we can simplify the stability factor $S(\mu, \eta)$: With $X_s(x) = X_s(\eta; \mu, x)$ we have
\begin{equation}
  \begin{aligned}
    S(\mu, \eta) & = \sup_{A} \frac{1}{|A|} \int_0^1 \int_{X_s^{-1}(A)} | \dot{X}_s(x) | \, dx \, ds \\
    & = \sup_{A} \int_0^1 \frac{1}{|A|} \int_{A} | \dot{X}_s(X_s^{-1}(y))| |\det D X_s^{-1}(y)| \, dy \, ds \\
    & = \int_0^1 \sup_{y \in \Omega} |\dot{X}_s(X_s^{-1}(y))| |\det D X_s^{-1}(y)| \, ds \\
    & = \int_0^1 \sup_{x \in \Omega} \frac{|\dot{X}_s(x)|}{|\det D X_s(x)|} \, ds.
  \end{aligned}
  \label{eq:stability-factor-det}
\end{equation}
Note the similarity to our informal calculation \eqref{eq:stability-informal}: For the nominator, we have $|\dot{X}_s(x)| \approx |X_0(x) - X_1(x)|$ and in $1d$ the denominator simplifies to $|X_s'(x)|$. In addition, because of the $s$ dependence, we now have some explicit control over the denominator by choosing a suitable homotopy.

In order to better understand the stability factor $S(\mu, \eta)$, let us first consider the integral bound $X_s(\eta; \mu, \cdot)^{-1}(A)$ in \eqref{eq:stability-factor}, or equivalently the denominator in \eqref{eq:stability-factor-det}. To this end, we consider a set $A \subset \Omega$ and observe how its volume changes when it is transformed along the flow $\eta \to A_\eta := X_s(\eta; \mu, A)$. We expect the stability factor to be small if the volume $|A_\mu| = |X_s(\eta; \mu, \cdot)^{-1}(A_\eta)|$ is not much larger that the volume $|A_\eta|$. This implies that contracting sets from $\eta \to \mu$ is expected to be stable, while growing them is probably unstable. 

To clarify this point, let us consider the example in Figure \ref{fig:jump-set} again. Say we have two parameters $\mu_c$, where the two jumps collide, and $\mu_2 < \mu_c$ for which we have two separate jumps. Let us first try to approximate $u(\cdot, \mu_c)$ form a transformation of $u(\cdot, \mu_2)$. If we choose any set $A_{\mu_2}$ between the two jumps, the transform indicated in the figure shrinks it to zero volume. By our analysis this is stable and intuitively all perturbations in the set $A_{\mu_2}$ are squashed to measure zero and thus irrelevant. Vice versa, using a transform of the snapshot $u(\cdot, \mu_c)$ to approximate $u(\cdot, \mu_2)$ is problematic. Our stability bound becomes large matching the intuition that the snapshot does not provide any information on how to fill the void between the two jumps of the target $u(\cdot, \mu_2)$.

Let us next consider the problem with the $L_\infty$-norm for a simple example problem. We choose $X_s(\eta; \mu, x) = X_0(\eta; \mu, x - s \epsilon v)$ for some direction $v$, so that the perturbation $X_1$ is a shift of $X_0$, which generally has large $L_\infty$ error if $X_0$ has sharp gradients or jumps. Note that with this definition we also have chosen a homotopy between the two transforms. Next, we assume that $\|D X_s(\eta; \mu, \cdot)^{-1}\|_2 \le c$ for some constant $c$. This allows the transform $X_s(\eta; \mu, \cdot)$ to have sharp gradients and places us in the situation where the sets $A_\eta$ grow at most moderately for which we expect stable behaviour by the discussion above. In the example of Figure \ref{fig:jump-set}, this includes the case of approximating $u(\cdot, \mu_c)$ form $u(\cdot, \mu_2)$ using a perturbed transform that misses the correct collision location in physical space by $\epsilon$.

Computing the stability factor \eqref{eq:stability-factor-det} yields
\[
  S(\mu, \lambda) = \int_0^1 \sup_{x \in \Omega} \frac{|\dot{X}_s(x)|}{|\det D X_s(x)|} \, ds
  \le \int_0^1 \sup_{x \in \Omega} s \epsilon \frac{|D X_s(x)|_2|v|}{|\det D X_s(x)|} \, ds.
\]
where $|\cdot|_2$ denotes the $\ell_2(\real^d)$ matrix norm. With the singular values $\sigma_1^s(x) \ge \cdots \ge \sigma_d^s(x) \ge c$ of $DX_s$, this simplifies to 
\[
  S(\mu, \lambda)
  = \int_0^1 s \epsilon |v| \sup_{x \in \Omega} \frac{\sigma_1^s(x)}{\prod_{i=1}^d \sigma_i^s(x)} \, ds 
  \le \epsilon |v| c^{1-d} \int_0^1 s \, ds = 
  \le \frac{1}{2} \epsilon |v| c^{1-d}.
\]
Thus, even if the $L_\infty$ error of the perturbed transforms is large, we may nonetheless have a small TSI error. Of course the result relies on a carefully chosen homotopy $X_s$. How to choose this in the general case is beyond the scope of this paper.

\section{Optimizing Singular Transforms}
\label{sec:optimize}

\subsection{Difficulties with the optimization of the transform}

In Section \ref{sec:ode}, we discussed singularities of the transforms at jump-set topology changes and efficient ways to approximate them. However, we still need a way to automatically calculate these transforms from the available data, i.e. the snapshots and training snapshots. As in \eqref{eq:opt}, we use an optimization problem 
\begin{align}
  & \min_{\Phi \in \transforms} \sigma(\Phi), &  \sigma(\Phi) := \|u(\cdot, \mu) - u_m(\cdot, \mu; \Phi)\|_{L_1(\Omega)},
  \label{eq:opt-ode}
\end{align}
where the notation $u_m(x, \mu; \Phi) = u_m(x,\mu)$ is used to emphasize the dependence of the TSI on the choice of the transport field $\Phi$ and $\transforms \subset \banach$ is an admissible set of transforms in a Banach space $\banach$ as discussed below. Note that the constraints $\transforms$ can be used to ensure that the resulting curves $X(\eta)$ do not leave the spacial domain $\Omega$ and that we only need to know $\Phi(\cdot, \eta)$ for finitely many $\eta \in \param_n$. We therefore only need a spacial discretization of $\Phi$ by selecting a basis, or more generally a dictionary, $\{\psi_i\}_{i \in \indexset}$ for some index set $\indexset$, with $\operatorname{span} \{\psi_i\}_{i \in \indexset} = \banach_n \subset \banach$ and optimize
\begin{align}
  & \min_{c \in \real^N} \boldsymbol{\sigma}(\boldsymbol{c}), &  \boldsymbol{\sigma}(\boldsymbol{c}) := \sigma \left( \sum_{i=1}^N c_i \psi_i \right).
  \label{eq:opt-ode-finite}
\end{align}
Instead of strictly enforcing that $x \to X(\eta;\mu, x)$ maps $\Omega$ to itself, we usually assume that $\nu \cdot \Phi(\cdot, \eta) = 0$, where $\nu$ is the outward normal, so that locally the curves $X(\eta; \mu, x)$ do not flow through the boundary. As we have seen above, already simple problems require sharp spacial gradients so that we want to choose $\banach_n$ with resolutions matching the high fidelity model of the snapshots. In this section, we discuss the implications for the optimization of \eqref{eq:opt-ode}. 

It has been shown in \cite{Welper2017} that we can expect the map $\boldsymbol{\sigma}$ to be Lipschitz continuous. Although more sophisticated algorithms for such non-smooth optimization problems are available \cite{Kiwiel1985,BurkeLewisOverton2005}, we start with a simple gradient descent method:
\begin{equation}
  \boldsymbol{c}^{n+1} = \boldsymbol{c}^n - \alpha \boldsymbol{\sigma}'(\boldsymbol{c}^n).
  \label{eq:grad-descent}
\end{equation}
Unfortunately, this simple choice does not provide satisfactory results. For example, Figure \ref{fig:simple-1d} shows the convergence behaviour for a very simple one dimensional test case: The gradient descent optimizer reduces the error a little bit but then stalls way before any reasonable alignment has been achieved. Since the example has no singularities, the figure also shows the results for a coarse discretization of $X(\eta; \mu, x)$ by second order polynomials in $x$. As we see, this coarse resolution does not pose a severe obstacle to the optimizer.

\subsection{Optimization in Banach Spaces}

To better understand the issue of the higher resolution, we first go all the way to  infinite resolution and optimize \eqref{eq:opt-ode} over the full function space $\banach$. We solve the optimization problem $\min_{\Phi \in \banach} \sigma(\Phi)$ by a descent method so that 
\[
  \Phi^{n+1} = \Phi^n - \alpha d^n
\]
for some descent direction $d^n$. Usually, we choose the gradient $d^n = \sigma'(\Phi)$, but this time we have to be more careful: Since we optimize in a Banach space, $\sigma'(\Phi)$ is a Fr\'{e}chet derivative and thus an element of the dual space $\banach'$ of $\banach$. If $\banach = \real^n$, we can easily identify $\left(\real^n \right)' = \real^n$, but for a generic Banach space $\banach$ this is not always possible. This implies that a term $\Phi^n - \alpha \sigma'(\Phi)$ does not necessarily make sense. 

Thus, the descent direction $d^n$ depends on the Banach space $\banach$. A good choice of $\banach$ should depend e.g. on the existence of minimizers or their stability, which is beyond the scope of this paper. Nonetheless, one can easily gain some insight about the space $\banach'$. Formally, the variation of the error in the direction $v \in \banach$ is given by
\[
  \partial_v \sigma(\Phi) = \int_\Omega \sign \big[u(x,\mu) - u_m(x,\mu; \Phi) \big] \partial_v u_m(x, \mu; \Phi) \, dx
\]
where using the definition of the TSI, we have
\[
  \partial_v u_m(x, \mu; \Phi) = \sum_{\eta \in \param_m} \ell_\eta(\mu) D u(X(\eta; \mu, x), \eta) DX(\eta; \mu, x) v
\]
Note in particular that this derivative contains the derivative $D u(x, \eta)$ with respect to $x$. Since $u$ has jumps we may assume that $u$ is of bounded variation so that $Du$ is a measure. At the jumps this measure has a singular component, e.g. a Dirac delta in $d=1$ or generally a measure supported on a $d-1$ dimensional manifold. This implies that the dual space $\banach'$ must be rich enough to contain these singular measures, which makes an identification of $\banach$ and $\banach'$ difficult.

Of course the optimization in Banach spaces is well understood and a descent direction $d^n$ can be constructed along the following lines, see e.g. \cite{Ulbrich2009}. First observe that the value of the objective function after one step is given by
\[
  \sigma(\Phi^{n+1}) = \sigma(\Phi^n) - \alpha \dualp{\sigma'(\Phi^n), d^n} + o(\alpha)
\]
so that we obtain an error reduction for sufficiently small $\alpha$ if $\dualp{\sigma'(\Phi^n), d^n} > 0$. For example, we can choose the Riesz representation of the Fr\'{e}chet defined by
\begin{align}
  d^n & = \|\sigma(\Phi^n)\|_{\banach'} w, & 
  w & = \operatorname{argmax}_{v \in \banach} \frac{\dualp{\sigma'(\Phi^n), v}}{\|v\|_\banach},
  \label{eq:riesz-descent}
\end{align}
so that $\dualp{\sigma'(\Phi^n), d^n} = \|\sigma'(\Phi^n)\|_{\banach'}^2$ and we obtain
\[
  \sigma(\Phi^{n+1}) = \sigma(\Phi^n) - \alpha \|\sigma'(\Phi^n)\|_{\banach'}^2 + o(\alpha).
\]
Recall that in our case the dual space $\banach$ should contain singular measures supported on $d-1$ dimensional manifolds. Via the trace theorem, measures of such type are contained in $H^{-1}(\Omega)$, which is a particularly simple choice because it is a Hilbert space and the computation of the Riesz representer reduces to solving the Poisson equation. More precisely, since $\Phi(\cdot, \eta): \Omega \to \real^d$ is vector valued, we choose $\banach = \prod_{\eta \in \param_m} \{v \in H^1(\Omega)^d : \, \nu \cdot v = 0 \text{ on } \partial \Omega\}$, where $\nu$ is the outward unit normal vector and the ``slip'' boundary conditions ensure that locally the resulting curves $X(\eta)$ do not leave $\Omega$. The product is used to obtain one function for each $\Phi(\cdot, \eta)$ with $|\param_n| = n$ values for $\eta$. Thus, the Riesz representation \eqref{eq:riesz-descent} is given by $-\Delta^{-1} \sigma'(\Phi^n)$ and we obtain the optimization scheme
\begin{equation}
  \Phi^{n+1} = \Phi^n + \alpha \Delta^{-1} \sigma'(\Phi^n),
  \label{eq:opt-laplace}
\end{equation}
where $-\Delta^{-1}$ solves the vector valued Poisson equation with the given boundary conditions. This is discretized by a Galerkin method in the obvious way. We will refer to this method as \emph{Laplace smoothing} in the following.

As an alternative to inverting a Laplacian, we can carefully choose the dictionaries $\{\psi_i\}_{i \in \indexset}$ for the discretization, similar to e.g. wavelet methods for elliptic equations \cite{Dahmen2003,Cohen2003}. To this end, let us assume that the $\psi_i$ are a frame for $\banach'$, so that
\[
  A \| \ell \|_{\banach'}^2 \le \sum_{i \in \indexset} \left| \dualp{\ell, \psi_i} \right| \le B \|\ell\|_{\banach'}^2
\]
for all $\ell \in \banach'$ and constants $A,B \ge 0$. Then, there is a dual frame $\{\bar{\psi}_i\}_{i \in \indexset}$ such that
\[
  B^{-1} \|\Phi\|_{\banach}^2 \le \sum_{i \in \indexset} \left| \dualp{\Phi, \bar{\psi}_i} \right| \le A^{-1} \|\Phi\|_{\banach'}^2
\]
for all $\Phi \in \banach$ and
\[
  \Phi = \sum_{i \in \indexset} \dualp{\Phi, \bar{\psi}_i} \psi_i.
\]
Note that $\Phi \in \banach$ if and only if the frame coefficients $c_i = \dualp{\Phi, \bar{\psi}_i}$ are in $\ell_2(\real^\indexset)$ so that in the original discretization \eqref{eq:opt-ode-finite} we now optimize over $\boldsymbol{c} \in \ell(\real^\indexset)$ for which we easily can identify the space with its dual. Indeed, for the resulting iteration 
\begin{align}
  \boldsymbol{c}^{n+1} & = \boldsymbol{c}^n - \alpha \boldsymbol{d}^n, & \boldsymbol{d}^n_i = \partial_i \boldsymbol{\sigma}(\boldsymbol{c}) = \dualp{\sigma'\left(\sum_{i=1}^\infty c_i \psi_i \right), \psi_i}
  \label{eq:opt-frame}
\end{align}
we have the error reduction
\begin{align*}
  \boldsymbol{\sigma}(\boldsymbol{c}^{n+1}) 
  & = \boldsymbol{\sigma}(\boldsymbol{c}^n) - \alpha \sum_{i=1}^\infty \left|\dualp{\sigma'\left(\sum_{i=1}^\infty c_i \psi_i \right), \psi_i}\right|^2 + o(\alpha) \\ 
  & \le \boldsymbol{\sigma}(\boldsymbol{c}^n) - \alpha A \left\| \sigma'\left(\sum_{i=1}^\infty c_i \right) \right\|_{\banach'}^2 + o(\alpha)
\end{align*}
for $\alpha$ sufficiently small. 

For simplicity, in the numerical experiments, we do not explicitly construct a frame, but rather define the transform as $X(\eta; \mu, x) = x + \sum_{\ell} (X^\ell(\eta; \mu, x) - x)$, where $X^\ell(\dots)$ are curves generated by transform fields $\Phi^\ell$ with spacial resolution level $\ell$ and $X^\ell(\cdots) - x$ the updates from the respective levels. This is referred to as \emph{multilevel smoothing} in the following.

\section{Numerical Experiments}
\label{sec:experiments}

In Section \ref{sec:optimizers-1d}, we first choose a simple problem without singularity to see how the optimizers of Section \ref{sec:optimize} perform in comparison to the original method recalled in Section \ref{sec:tsi-review}. Then, in Section \ref{sec:collision-1d}, we consider a test with a singularity where the original low resolution transforms fail. In Section \ref{sec:curve-1d}, we compare different orders of the transforms and in Section \ref{sec:numerical-2d-example} we consider a $2d$ example.

In all examples, the finite element discretizations of the snapshots and transform fields are handled by the Deal.II library \cite{BangerthHartmannKanschat2007} and all TSI components are then put together in Tensorflow \cite{AbadiAgarwalBarhamEtAl2015}. This setup allows all gradients of the error to be automatically computed by the back-propagation algorithm in Tensorflow. 

\subsection{Testing the Optimizers}
\label{sec:optimizers-1d}

In this section, we first consider a very simple example in order to see how the optimizers of Section \ref{sec:optimize} applied to high resolution transforms of Section \ref{sec:ode} perform in comparison to a simple gradient descent method and low resolution transforms as described in \cite{Welper2017} or Section \ref{sec:tsi-review}. To this end, we consider
\begin{align*}
  u(x, \mu) & = \left\{ \begin{array}{ll}
    1 & x \le \mu \\ -1 & x > \mu
  \end{array} \right. &
  & \begin{aligned}
    \param_m = \param_n & = \{-0.2, 0.2\} \\
    \param_T & = \{0\}
  \end{aligned}& 
  \Omega & = [-1.5, 1.5].
\end{align*}
The snapshots and high resolution transforms are are discretized by piecewise linear functions on $32$ cells and the low resolution transforms are second order polynomials. We use five fixed point iterations \eqref{eq:fixed-point} and no scaling \eqref{eq:scaling}.

Since the objective function typically has a kink at the minimum, its value from gradient descent steps typically oscillates near the minimum. For all optimization methods in this section, we select a learning rate that shows slight oscillations at the end to make them comparable. This is not very accurate but sufficient to demonstrate the main point of this experiment: Despite the much higher number of degrees of freedom, the Laplace and multilevel smoothing result in a convergence behaviour that is comparable to the original gradient descent method applied to low resolution transforms in \cite{Welper2017}. 

The concrete results are reported in Figure \ref{fig:simple-1d}. First observe that a simple gradient descent method without smoothing for high resolution transforms stalls way before it reaches an acceptable minimum. Second, the value of the objective function never converges to zero. Instead, the TSI training error typically saturates at a level comparable to the discretization error of the snapshots. This effect will be observed for all numerical experiments in this paper and is discussed more thoroughly in \cite{Welper2017}. 

\pgfplotsset{axis function/.style={
  height=5cm, 
  yticklabel pos=right, 
  ticklabel style={font=\tiny}, 
  y label style={at={(axis description cs:0.2,0.5)}}
}}

\newcommand{\AddPlotTsiSimple}[3]{
      \addplot[mark=none, thick, #1] table[x=x, y=#2, col sep=comma] {pics/#3.csv};
}

\newcommand{\AddPlotErrors}[3]{
      \addplot[mark=none, thick, #1] table[x=n, y=#2, col sep=comma] {pics/#3.csv};
}

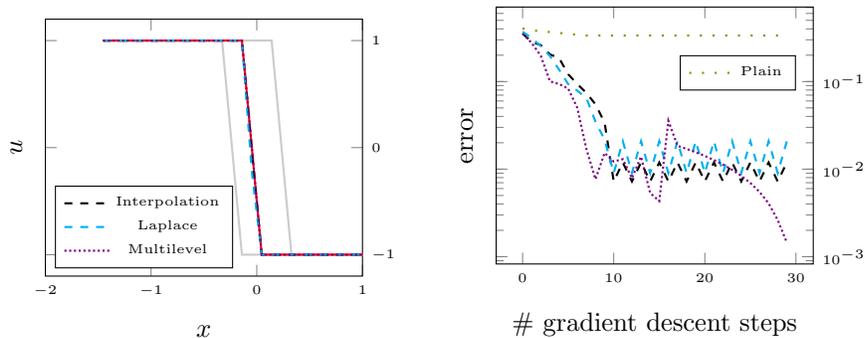
\begin{figure}[htb]

  \hfill
  \begin{tikzpicture}
    \begin{axis}[
      axis function, 
      xmin=-2, xmax=1,
      xlabel=$x$, 
      ylabel=$u$,
      legend style={legend pos=south west, font=\tiny}
    ] 

      \AddPlotTsiSimple{color=black!20}{snapshot 0}{simple}
      \AddPlotTsiSimple{color=black!20}{snapshot 1}{simple}
      \AddPlotTsiSimple{color=red}{exact}{simple}
      \AddPlotTsiSimple{color=black, dashed}{interpol}{simple}
      \AddPlotTsiSimple{color=cyan, dashed}{laplace}{simple}
      \AddPlotTsiSimple{color=violet, densely dotted}{multilevel}{simple}
      \legend{,,,Interpolation, Laplace,Multilevel}

    \end{axis}
  \end{tikzpicture}
  \hfill 
  \begin{tikzpicture}
    \begin{axis}[
      axis function, 
      xlabel=$\#$ gradient descent steps, 
      ylabel=error,
      ymode=log,
      legend style={at={(axis cs: 30,0.2)}, anchor=north east, font=\tiny}
    ] 

      \AddPlotErrors{color=black, dashed}{interpol}{simple_errors}
      \AddPlotErrors{color=olive, loosely dotted}{plain}{simple_errors}
      \AddPlotErrors{color=cyan, dashed}{laplace}{simple_errors}
      \AddPlotErrors{color=violet, densely dotted}{multilevel}{simple_errors}
      \legend{,Plain,,}

    \end{axis}
  \end{tikzpicture}

  \caption{Results for the $1d$ Example in Section \ref{sec:optimizers-1d}. Left: Snapshots (gray), exact solution (solid red) and reconstructions at $\mu = 0$: TSI with Laplace smoothing (cyan, dashed), multilevel smoothing (violet, dotted) and a TSI with low resolution transform \eqref{eq:transform-interpol} (black, dashed). Right: corresponding gradient descent errors with the same color coding. In addition the figure shows the errors for high resolution transforms without smoothing (olive dotted).}

  \label{fig:simple-1d}

\end{figure}

\subsection{Testing Transform Resolutions}
\label{sec:collision-1d}

Next, we test high versus low resolution transforms near a singularity. We choose the example
\begin{align*}
  u(x, \mu) & = \left\{ \begin{array}{rl} 
    \frac{x-l}{-(1-\mu) - l} & x \le -(1 - \mu) \\
    \frac{x-r}{ (1-\mu) - r} & x \ge \phantom{-(} 1 - \mu \phantom{)} \\
    0 & \text{else}
  \end{array}\right., &
  & \begin{aligned}
    \param_m & = \{0.6\} \\
    \param_n & = \{0.6\} \\
    \param_T & = \{0.9\}
  \end{aligned}& 
  & \begin{aligned}
  \Omega & = [-1.5, 1.5] \\
         & = [l, r]
  \end{aligned},
\end{align*}
with snapshots and high resolution transforms discretized by piecewise linear functions on $128$ cells. For the low resolution transforms we choose $x \to X(\eta; \mu, x)$ as third order polynomials. They are defined by four values: Two are used to ensure that the boundary positions do not move and two are optimized to track the two jumps. We only compare the low and high resolution transforms at the training parameter $\param_T$ because by \eqref{eq:transform-interpol} the low order transform is independent of the target $\mu$ for reconstructions from one single snapshot. For the high resolution transforms, we use three fixed point iterations in \eqref{eq:fixed-point} and scaling $s \in \{0.7^i: i=0 \dots 3\}$ in increasing size.

In Figure \ref{fig:1d-solutions} left, we see the TSI reconstructions at $\param_T$. We see that the high resolution transforms provide a good reconstruction, whereas the low resolution transforms do not. We have chosen the degree of the low resolution transforms sufficiently high so that in principle they can align the jumps. However, this requires a rather sharp gradient, which would create a big oscillation in the polynomial. As a result, the composition $x \to u(X(\eta; \mu, x), \eta)$ cannot be linear in the regions where is exact solution $x \to u(\mu, x)$ is. This can be seen by carefully looking at the plot in the linear regions. Thus, the optimizer must balance the fit in the linear regions with the aliment, which is not possible to do accurately. In contrast, by Figure \ref{fig:1d-solutions} right, the high resolution transforms are almost piecewise linear with kinks at the jump locations. Therefore they can simultaneously align the jumps and provide an accurate fit in the linear regions. The effect can also be clearly seen in the training errors in Figure \ref{fig:errors} left, where the low resolution transforms saturate much earlier than the high resolution transforms. 

\begin{figure}[htb]

  \hfill
  \begin{tikzpicture}
    \begin{axis}[
      axis function, 
      xlabel=$x$, 
      ylabel=$u$,
      legend style={legend pos=south east, font=\tiny}
    ] 

      \AddPlotTsiSimple{color=black!20}{snapshot 0}{collision}
      \AddPlotTsiSimple{color=red}{exact}{collision}
      \AddPlotTsiSimple{color=black, dashed}{interpol}{collision}
      \AddPlotTsiSimple{color=cyan, dashed}{laplace}{collision}
      \AddPlotTsiSimple{color=violet, densely dotted}{multilevel}{collision}
      \legend{,,Interpol,Laplace,Multilevel}

    \end{axis}
  \end{tikzpicture}
  \hfill 
  \begin{tikzpicture}
    \begin{axis}[
      axis function, 
      xlabel=$x$, 
      ylabel=$X(\eta)$,
      legend style={legend pos=south east, font=\tiny}
    ] 

      \AddPlotTsiSimple{color=black, dashed}{interpol xt 0}{collision}
      \AddPlotTsiSimple{color=cyan, dashed}{laplace xt 0}{collision}
      \AddPlotTsiSimple{color=violet, densely dotted}{multilevel xt 0}{collision}
      \legend{Interpol,Laplace,Multilevel}

      \newcommand{\JumpSourceZero}{-0.4}
      \newcommand{\JumpSourceOne}{0.4}
      \newcommand{\JumpTargetZero}{-0.1}
      \newcommand{\JumpTargetOne}{0.1}

      \newcommand{\PlotMu}{0.7}
      \newcommand{\PlotEtaOne}{0.55}
      \newcommand{\PlotEtaTwo}{0.85}
      \newcommand{\PlotJumpOne}[1]{#1-1}
      \newcommand{\PlotJumpTwo}[1]{1-#1}
      \draw[red, dotted] (axis cs: \JumpTargetZero, -1.5) -- (axis cs: \JumpTargetZero, 1.5);
      \draw[red, dotted] (axis cs: \JumpTargetOne, -1.5) -- (axis cs: \JumpTargetOne, 1.5);
      \draw[red, dotted] (axis cs: -1.5, \JumpSourceZero) -- (axis cs: 1.5, \JumpSourceZero);
      \draw[red, dotted] (axis cs: -1.5, \JumpSourceOne) -- (axis cs: 1.5, \JumpSourceOne);

    \end{axis}
    \hfill~
  \end{tikzpicture}

  \caption{Results for the Example in Section \ref{sec:collision-1d}. Left: Snapshots (gray), exact function at $\param_T$ (red) and TSI with low resolution (black dashed) and high resolution with Laplace smoothing (cyan, dashed) and multilevel smoothing (violet, dotted). Right: The corresponding transform with the same color coding. The red lines are the locations of the jumps of the training snapshot (vertically) and the target snapshots (horizontally).}

  \label{fig:1d-solutions}

\end{figure}
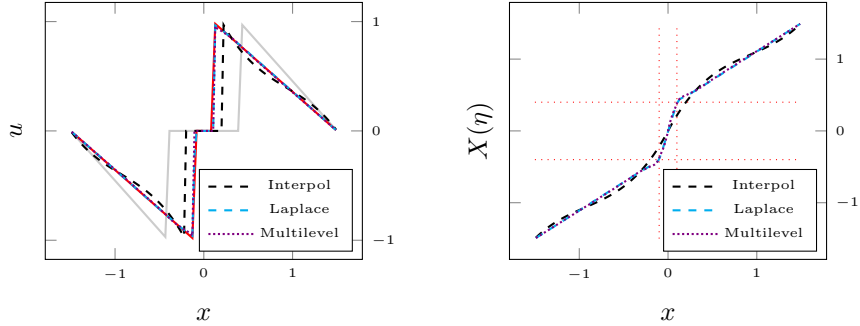

\begin{figure}[htb]

  \hfill
  \begin{tikzpicture}
    \begin{axis}[
      axis function, 
      xmin=0, xmax=300,
      xlabel=$\#$ gradient descent steps, 
      ylabel=error,
      ymax=3,
      ymode=log,
      legend style={legend pos=north east, font=\tiny}
    ] 

      \AddPlotErrors{color=black, dashed}{interpol}{collision_errors}
      \AddPlotErrors{color=cyan, dashed}{laplace}{collision_errors}
      \AddPlotErrors{color=violet, densely dotted}{multilevel}{collision_errors}
      \legend{Interpol,Laplace,Multilevel}

    \end{axis}
  \end{tikzpicture}
  \hfill
  \begin{tikzpicture}
    \begin{axis}[
      axis function, 
      xlabel=$\#$ gradient descent steps, 
      ylabel=error,
      ymode=log,
    ] 

      \AddPlotErrors{color=cyan, dashed}{laplace}{ellipse_errors}

    \end{axis}
  \end{tikzpicture}
  \hfill~

  \caption{Gradient descent errors for the examples in Sections \ref{sec:collision-1d} and \ref{sec:numerical-2d-example}. The left picture uses the color coding of Figure \ref{fig:1d-solutions}.}
  \label{fig:errors}

\end{figure}
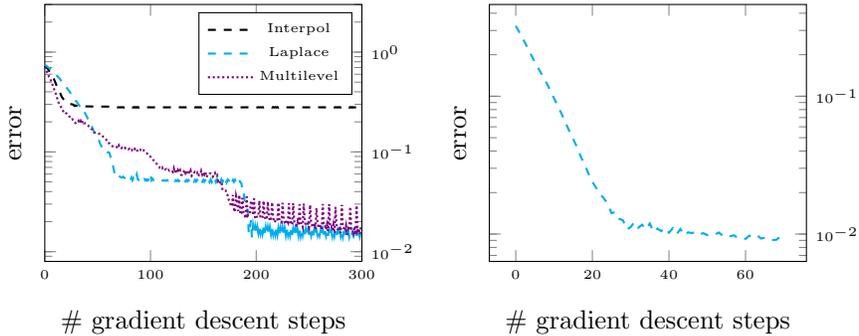

\subsection{Testing Transform Orders}
\label{sec:curve-1d}

In this section, we compare different polynomial orders in $\eta$ of the transforms. To this end, we choose the example
\begin{align*}
  u(x, \mu) & = \left\{ \begin{array}{rl} 
    \frac{x-l}{-p(\mu) - l} & x \le -p(\mu) \\
    \frac{x-r}{ p(\mu) - r} & x \ge \phantom{-} p(\mu) \\
    0 & \text{else}
  \end{array}\right., &
  & \begin{aligned}
    \param_m & = \{0.6\} \\
    \param_n & = \{0.6,0.7\} \\
    \param_T & = \{0.8, 0.96\}
  \end{aligned}& 
  & \begin{aligned}
  \Omega & = [-1.5, 1.5] \\
         & = [l, r]
  \end{aligned},
\end{align*}
where $p(\mu)$ is a second order polynomial given by $p(0.6) = 0.4$, $p(0.8) = 0.3$ and $p(1) = 0$. All the other parameters are set as in the example of Section \ref{sec:collision-1d}. Note in particular that the jump-set is now a parabola in parameter so that a first order reconstruction as in the last section can no longer accurately recover $u$. For a comparison, we choose Laplace smoothing once with first order transforms (in $\eta$) with $\param_n = \{0.6\}$ and once with second order with $\param_n = \{0.6,0.7\}$. To be fair, both are trained on the same training set $\param_T$ as given above.

The results for $\mu = 0.86$ are shown in Figure \ref{fig:curve}. In the function plot the second order reconstruction looks marginally better, but a look at the training errors reveals that they are better by a factor of about $6$. The final training error for the second order case is about twice as much as for the example of Section \ref{sec:collision-1d}. This is encouraging because the profiles are identical, only the dynamic in parameter is changed, and in the second order example we add the training error at two training parameters, whereas we only use one in the previous example.

In principle, at this place we should present some plots of convergence rates. However, these are difficult to achieve because of the error saturation. Even if we achieve high order reconstructions, the errors need to be matched by the discretizations of the snapshots in physical space. Since this is quite difficult because of the jumps, we omit a more thorough convergence analysis.

\begin{figure}[htb]

  \hfill
  \begin{tikzpicture}
    \begin{axis}[
      axis function, 
      xlabel=$x$, 
      ylabel=$u$,
      legend style={legend pos=south east, font=\tiny}
    ] 

      \AddPlotTsiSimple{color=black!20}{snapshot 0}{curve}
      \AddPlotTsiSimple{color=red}{exact}{curve}
      \AddPlotTsiSimple{color=black, dashed}{order 1}{curve}
      \AddPlotTsiSimple{color=cyan, dashed}{order 2}{curve}
      \legend{,,Order 1, Order 2}

    \end{axis}
  \end{tikzpicture}
  \hfill 
  \begin{tikzpicture}
    \begin{axis}[
      axis function, 
      xmin=0, xmax=300,
      xlabel=$\#$ gradient descent steps, 
      ylabel=error,
      ymode=log,
      legend style={legend pos=north east, font=\tiny}
    ] 

      \AddPlotErrors{color=black, dashed}{order 1}{curve_errors}
      \AddPlotErrors{color=cyan, dashed}{order 2}{curve_errors}
      \legend{Order 1, Order 2}

    \end{axis}
  \end{tikzpicture}

  \caption{Results for the Example in Section \ref{sec:curve-1d}. Left: Snapshots (gray), exact function at $\param_T$ (red) and TSI reconstruction with linear in $\eta$ (black dashed) and quadratic (cyan dashed) transforms. Right: The corresponding gradient descent errors.}

  \label{fig:curve}

\end{figure}
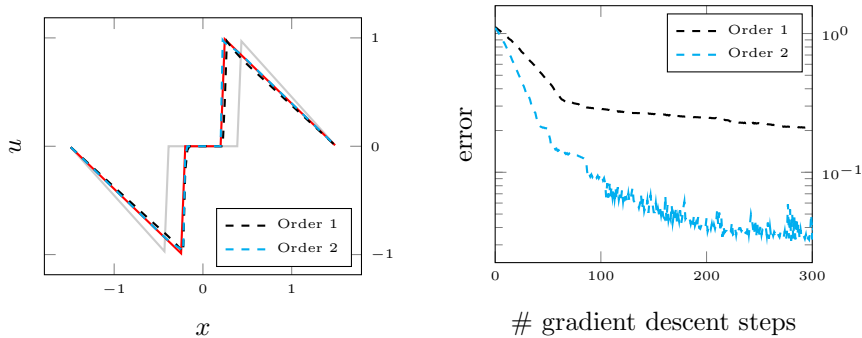

\subsection{$2d$ Example}
\label{sec:numerical-2d-example}

Finally, we consider one example in two spacial dimensions. We use the function
\begin{align*}
  u(x, y, \mu) & = \left\{ \begin{array}{rl} 
    1 & \left( \frac{r}{2\mu} \right)^2 + \left( \frac{s}{\mu} \right)^2 \le 1\\
    0 & \text{else}
  \end{array}\right., &
  \begin{pmatrix}
    r \\ s
  \end{pmatrix} & = \begin{pmatrix}
    \cos(1.5 \mu) & -\sin(1.5 \mu) \\
    \sin(1.5 \mu) &  \cos(1.5 \mu)
  \end{pmatrix}
\end{align*}
which is a rotated and scaled ellipse, with parameters
\begin{align*}
  & \begin{aligned}
    \param_m & = \{c(1), c(0.3)\} \\
    \param_n & = \{c(1)\} \\
    \param_T & = \{c(0.6), c(0)\}
  \end{aligned}& 
  c(\lambda) & = 0.2 \lambda + 0.05 (1-\lambda) & 
  \Omega & = [-1, 1]^2.
\end{align*}
The snapshots and high resolution transforms are discretized by piecewise bilinear functions on a rectangular gird with $64 \times 64$ cells. We use three fixed point iterations in \eqref{eq:fixed-point} and scaling $s \in \{0.7^i: i=0 \dots 2\}$ in increasing size.

Figure \ref{fig:2d} shows a solution plot for $\mu = c(0.25)$ together with some of the curves $\eta \to X(\eta; \mu, x)$ originating on the boundary of the correct ellipse. Figure \ref{fig:errors} right contains the corresponding training errors. We see that also in this $2d$ example, the method can achieve a good alignment with a moderate number of gradient descent steps.

As a final remark, note that for all experiments we use a fixed step size of a conservative magnitude. This is necessary because of the non-smooth nature of the objective function. Normally, near a minimum, the gradient becomes small so that the update between two gradient descent steps converges to zero. In our case however, we expect a kink at the minimum, similar to minimizing an absolute value $|x|$. Therefore, the gradient does not become small, which is compensated by a small step size. The effect can be seen in the convergence plots, where in the end the error oscillates around an equilibrium with amplitude depending on the step size. On the other hand, far away from the minimum, the conservative choice of the step size hinders a faster convergence.

\newcommand{\AddPlotTsiDimTwo}[1]{
   \addplot3[surf] table[x=x, y=y, z=#1, col sep=comma] {pics/ellipse.csv};
}
\newcommand{\AddPlotTsiCurves}[1]{
  \foreach \i in {0,1,...,8}{
    \addplot[mark=none, thick] table[x=#1 x\i, y=#1 y\i, col sep=comma] {pics/ellipse_xt.csv};
  \addplot[mark=x] table[x=initial x\i, y=initial y\i, col sep=comma] {pics/ellipse_xt.csv};
  }
}

\newcommand{\PlotTwoD}[1]{
  \begin{tikzpicture}[scale=0.5]
    \begin{axis}[
        xticklabels={,,},
        yticklabels={,,},
        view={0}{90},
      ]
      #1
    \end{axis}
  \end{tikzpicture}
}

\begin{figure}[htb]
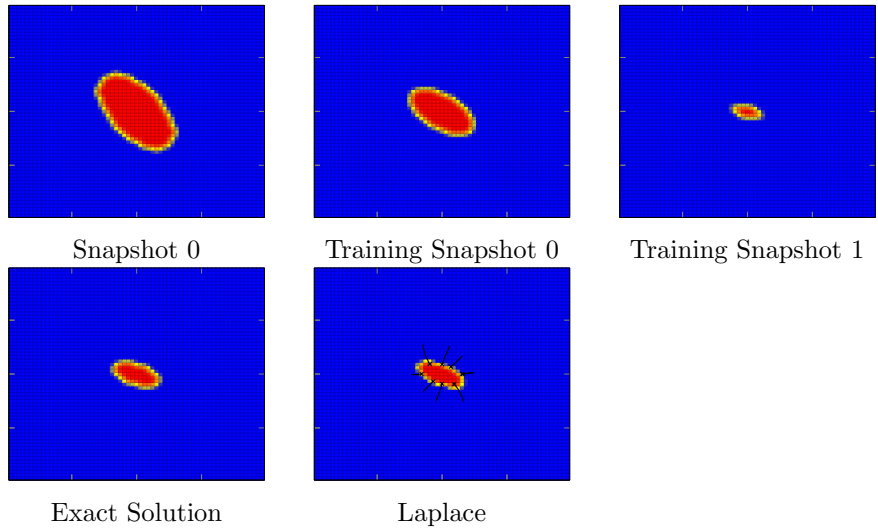


  \begin{tabular}{ccc}
    \PlotTwoD{\AddPlotTsiDimTwo{snapshot 0}} & 
    \PlotTwoD{\AddPlotTsiDimTwo{train 0}} & 
    \PlotTwoD{\AddPlotTsiDimTwo{train 1}} 
    \\
    Snapshot $0$ & Training Snapshot $0$ & Training Snapshot $1$
    \\
    \PlotTwoD{\AddPlotTsiDimTwo{exact}} &
    \PlotTwoD{\AddPlotTsiDimTwo{laplace}\AddPlotTsiCurves{laplace}} &
    \\
    Exact Solution & Laplace &
  \end{tabular}

  \caption{Results for numerical tests in Section \ref{sec:numerical-2d-example}. The top row contains the input data and the bottom row the exact solution and TSI reconstruction. The black lines in the last plot indicate several transforms $\eta \to X(\eta; \mu, x)$ with initial values $x$ at the respective black $\times$s.}

  \label{fig:2d}

\end{figure}

\begin{appendices}

\section{Proof of Proposition \ref{prop:stab}}
\label{appendix-proof-prop}

The following proof of Proposition \ref{prop:stab} is almost verbatim from \cite[Section 3]{Welper2017}. 

Let us first estimate $\|v \circ X_0 - v \circ X_1\|_{L_1(\Omega)}$ for $X_s(x) := X_s(\eta, \mu, x)$ for some fixed $\mu$ and $\eta$ and some function $v \in C^1(\Omega)$. Applying the fundamental theorem for line integrals, we obtain
\begin{equation*}
  (v \circ X_0)(x) - (v \circ X_1)(x) = \int_0^1 v'(X_s(x)) \dot{X}_s(x) \, \text{d} s.
\end{equation*}
Note that we can regard $|\dot{X}(s)|$ as a density and thus define a transition kernel $\tau(s,A) = \int_A |\dot{X}_s(x)| \, dx$. Then, with the pushforward $\kappa(s, A) =  (X_s)_* \tau(s, A) = \tau(s, X_s^{-1}(A))$ and the Lebesgue measure $\lambda_1$ on the real line, the definition of $S(\mu, \lambda)$ implies
\[
  (\lambda_1 \otimes \kappa)([0,1] \times A) = \int_0^1 \kappa(s, A) \, ds = \int_0^1 \int_{X_s^{-1}(A)} |\dot{X}_s(x)| \, dx \, ds \le S(\mu, \eta) |A|.
\]
Thus, using the Fubini theorem for transition kernels
we obtain 
\begin{multline}
  \| v \circ X_0 - v \circ X_1 \|_{L_1(\Omega)} \\
  \begin{aligned}
    & \le \int_\Omega \int_0^1 |v'(X_s(x))| |\dot{X}_s(x)| \, \text{d} s \, \text{d} x
    = \int_0^1 \int_\Omega |v'(X_s(x))| \tau(s, dx) \, \text{d} s \\
    & = \int_0^1 \int_\Omega |v'(y)| \kappa(s, dy) \, \text{d} s
    = \int_{[0,1] \times \Omega} |v'(y)| d (\lambda_1 \otimes \kappa)(s,y) \\
    & \le S(\mu, \eta) \int_\Omega |v'(y)|\, \text{d} y
    = S(\mu, \eta) \|v\|_{BV(\Omega)}.
  \end{aligned}
  \label{eq:stability-single}
\end{multline}
Next, we extend this estimate from functions $v \in C^1(\Omega)$ to functions $v \in BV(\Omega)$ by a density argument. To this end, note that for all $\epsilon > 0$ there is a $v_\epsilon \in C^1(\Omega)$ such that
\begin{align*}
  \|v - v_\epsilon\|_{L_1(\Omega)} & \le \epsilon & \|v_\epsilon'\|_{L_1(\Omega)} & \le \|v\|_{BV(\Omega)} + \epsilon.
\end{align*}
Thus, to apply a density argument, is suffices to bound $\|v \circ X_0 - v_\epsilon \circ X_0\|_{L_1(\Omega)}$ and $\|v \circ X_1 - v_\epsilon \circ X_1\|_{L_1(\Omega)}$. Since by assumption $(X_s)_* \lambda (A) \le B \lambda(A)$ for all measurable sets $A$, we conclude that 
\begin{align*}
  \|v \circ X_0 - v_\epsilon \circ X_0\|_{L_1(\Omega)}
  & = \int_\Omega |v(y) - v_\epsilon(y)| \, \text{d} (X_0)_* \lambda(y) \\
  & \le B \|v - v_\epsilon\|_{L_1(\Omega)}
\end{align*}
The bound for $\|v \circ X_1 - v_\epsilon \circ X_1\|_{L_1(\Omega)}$ follows analogously.

Finally, recalling that $X_s(x) = X_s(\eta; \mu, x)$, we apply the estimate \eqref{eq:stability-single} to the full TSI and obtain
\begin{multline*}
  \|u_m(\cdot, \mu; X_0) - u_m(\cdot, \mu, X_1)\|_{L_1(\Omega)} \\
  \begin{aligned}
    & \le \sum_{\eta \in \param_m} |\ell_\eta(\mu)| \| u(X_0(\eta; \mu, \cdot), \eta) - u(X_1(\eta; \mu, \cdot), \eta) \|_{L_1(\Omega)} \\
    & \le \Lambda_n \max_{\eta \in \param_m} \| u(X_0(\eta; \mu, \cdot), \eta) - u(X_1(\eta; \mu, \cdot), \eta) \|_{L_1(\Omega)} \\
    & \le \Lambda_n \max_{\eta \in \param_m} \big[ S(\mu, \eta) \|u(\cdot, \eta)\|_{BV(\Omega)} \big].
  \end{aligned}
\end{multline*}

\end{appendices}

\bibliographystyle{abbrv}
\bibliography{tsi-high-res}

\end{document}